\documentclass[reqno]{amsart}

\title[A metrization theorem for edge ends]{A metrization theorem for edge-end spaces of infinite graphs}
\author{Max Pitz}
\address{Universit\"at Hamburg, Department of Mathematics, Bundesstrasse 55 (Geomatikum), 20146 Hamburg, Germany}
\email{max.pitz@uni-hamburg.de}
\usepackage{amsmath,amssymb,amsthm}  
\usepackage{mathtools}
\usepackage{tikz, tikz-cd}
\usetikzlibrary{shapes.geometric}
\usetikzlibrary{shapes,decorations.pathmorphing,backgrounds,positioning,fit,petri, hobby}
\usepackage{comment}
\usepackage{enumerate}
\usepackage{enumitem}
\usepackage{mathrsfs}

\usepackage{thmtools, thm-restate}

\usepackage{xcolor} 	
\usepackage[unicode]{hyperref}
\hypersetup{
	colorlinks,
    linkcolor={red!60!black},
    citecolor={green!60!black},
    urlcolor={blue!60!black},
}
\usepackage[abbrev, msc-links]{amsrefs} 

\usepackage[utf8]{inputenc}
\usepackage[T1]{fontenc}
\usepackage{lmodern}
\usepackage[babel]{microtype}
\usepackage[english]{babel}

\usepackage{enumitem}

\let\polishlcross=\l
\def\l{\ifmmode\ell\else\polishlcross\fi}

\let\emptyset=\varnothing

\let\eps=\varepsilon
\let\theta=\vartheta
\let\rho=x
\let\phi=\varphi

\def\NN{\mathbb N}

\def\cC{{\mathscr C}}

\def\cP{{\mathcal P}}

\def\cX{{\mathcal X}}
 
\newcommand{\Set}[1]{{\left\lbrace {#1} \right\rbrace}}

\def\set#1:#2{\Set{{#1} \colon {#2}}}

\newcommand{\uc}[1]{\lfloor #1\rfloor}

\renewcommand{\subset}{\subseteq}
\theoremstyle{plain}

\newtheorem{thm}{Theorem}[section]

\newtheorem{clm}[thm]{Claim}

\newtheorem{lemma}[thm]{Lemma}

\newtheorem{problem}{Problem}

\theoremstyle{definition}

\begin{document}

\begin{abstract}
We prove that the edge-end space of an infinite graph is metrizable if and only if it is first-countable. This strengthens a recent result by Aurichi, Magalhaes Jr.\ and Real (2024).

Our central graph-theoretic tool is the use of tree-cut decompositions, introduced by Wollan (2015) as a variation of tree decompositions that is based on edge cuts instead of vertex separations. In particular, we give a new, elementary proof for Kurkofka's result (2022) that every infinite graph has a tree-cut decomposition of finite adhesion into its $\omega$-edge blocks. Along the way, we also give a new, short proof for a classic result by Halin (1984) on $K_{k,\kappa}$-subdivisions in $k$-connected graphs, making this paper self-contained.
\end{abstract}



\maketitle

\section{Introduction}

When studying infinite graphs $G$, both abstract graphs as well as geometric, hyperbolic graphs, one is often interested in the `boundary of $G$ at infinity'. These boundaries are formalised by considering certain equivalence relations on the \emph{rays} of $G$ (the $1$-way infinite paths in $G$). For abstract graphs, the two most common equivalence relations are as follows: 

Following Halin \cite{Halin_Enden64}, two rays in a graph $G = (V,E)$ are \emph{vertex-equivalent} if no finite set of vertices separates them. The resulting equivalence classes of rays are the \emph{vertex-ends} of $G$, and
the set of all ends is denoted by $\Omega(G)$. 
The term ``boundary at infinity'' is justified by a natural (Hausdorff, but not necessarily compact) topology on the space  $|G|= V \cup \Omega(G)$ in which every converges to `its' end, see \S\ref{subsec_background} below for details. With the subspace topology, $\Omega(G)$ becomes the \emph{end space} of $G$. We refer the reader to Diestel's survey articles \cite{diestel2011locally, diestel2010locally} for a number of applications of this topological viewpoint.

Following Hahn, Laviolette and {\v{S}}ir{\'a}{\v{n}} \cite{hahn1997edge}, two rays in a graph $G = (V,E)$ are \emph{edge-equivalent} if no finite set of edges separates them. The corresponding equivalence classes of rays are the \emph{edge-ends} of $G$, and
the set of all edge-ends is denoted by $\Omega_E(G)$. 
Once more, we have a natural topology on the space  $\|G\|= V \cup \Omega_E(G)$ in which every end lies in the closure of any of its representative rays. This topology is not necessarily Hausdorff, but if $G$ is connected, as we assume, then it is compact. See again \S\ref{subsec_background} for details.
With the subspace topology, $\Omega_E(G)$ becomes the \emph{edge-end space} of $G$. Edge-end spaces have recently been investigated in \cite{aurichi2025edge,aurichi2024topological}.

Vertex-equivalent rays are also edge-equivalent, and if $G$ is locally finite or if $G$ is a tree, then also the converse implication holds, and $\Omega(G)$ and $\Omega_E(G)$ are in fact identical (even as topological spaces). However, in graphs that contain vertices of infinite degree, edge-equivalent rays are not necessarily vertex-equivalent. Thus, vertex-equivalence is generally a finer relation than edge-equivalence, and consequently, the topological spaces $\Omega(G)$ and $\Omega_E(G)$ may differ. And while the space of vertex-ends $\Omega(G)$ is topologically well understood, much less is known about edge-end spaces. For example, we know precisely under which conditions $|G|$ and $\Omega(G)$ are metrizable \cite{diestel2006end, kurkofka2021approximating}, but for edge-end spaces, no exact characterisation has been known. Our first main result resolves this problem:

\begin{restatable}{theorem}{metrization}
\label{thm_met}
The following properties are equivalent for an edge-end space $X$: 
\begin{enumerate}
\item  $X$ is first countable,
\item  $X$ is metrizable,
\item  $X$is completely ultrametrizable,
\item $X$ is homeomorphic to the end-space $\Omega(T) = \Omega_E(T)$ of a tree $T$.
\end{enumerate}
\end{restatable}

This strengthens a recent result by Aurichi, Magalhaes Jr. and Real \cite[Theorem 4.5]{aurichi2024topological} who established that first-countable, Lindel\"of edge-end spaces are metrisable.
 
The interesting implication in Theorem~\ref{thm_met} is $(1)\Rightarrow (4)$, with the other implications $(4) \Rightarrow (3)  \Rightarrow (2)  \Rightarrow (1)$ being trivial or well known. That `first countable' implies `metrizable' is a surprising local-to-global phenomenon, which is usually encountered only in spaces with much richer structure such as topological groups \cite{hewitt2013abstract}. In order to prove $(1)\Rightarrow (4)$, all the hard work lies in proving the following representation theorem for \emph{all} edge-end spaces, whether first-countable or not: 

\begin{restatable}{theorem}{representation}\label{thm_rep}
Up to homeomorphism, the edge-end spaces are precisely the subspaces $X\subseteq \| T\|$ with $\Omega_E(T)\subseteq X$, for a graph-theoretic tree $T$.
\end{restatable}

Theorem~\ref{thm_met} follows by slightly modifying the tree $T$ in Theorem~\ref{thm_rep} to a tree $T'$ (using the assumption of first countability) so that every $x \in X \cap V(T)$ is represented by a ray $x \in X \cap \Omega(T')$; see \S\ref{subsec_met}.

We prove Theorem~\ref{thm_rep} in two steps:  First, as our main graph-theoretic tool we use \emph{tree-cut decomposition} as introduced by P.~Wollan in \cite{WOLLAN201547}. Tree-cut decomposition of finite graphs have recently emerged for their algorithmic applications  \cite{ganian2015algorithmic, giannopoulou2021linear, kim2018fpt}, but also for their structural properties, see e.g.\ \cite{GIANNOPOULOU20211} and the references therein. 
For infinite graphs, they have proven to be instrumental in detemining the minor- and immersion minimal infinitely connected graphs \cite{kurkofka2020every,knappe2024immersion}. 
In Section~\ref{sec_2}, we construct a certain tree-cut decomposition of the underlying graph $G$ which essentially captures all finite edge cuts in the graph $G$ simultaneously (see Theorem~\ref{thm_treebond_finiteadhesion} for details). 
That such a tree-cut decomposition exists is a recent result by J.~Kurkofka \cite[Theorem 5.1]{kurkofka2020every}; our contribution here is to give a new, elementary proof of this result. 
To make this part of the argument self-contained, we additionally provide an elementary proof for a classic result by Halin from \cite{halin1978simplicial} that every uncountable $k$-connected graph contains a subdivision of the complete bipartite graph $K_{k,\aleph_1}$ using nothing but Zorn's lemma, which may be of independent interest. 
Having found the suitable candidate for $T$, we derive in Section~\ref{sec_3}  the topological implications announced above.
We conclude with the following natural open problem:

\begin{problem}
Find a purely topological characterisation of edge-end spaces.
\end{problem}

For end spaces $\Omega(G)$, this has been achieved in \cite{pitz2023characterising}. However, by the main result of Aurichi, Magalhaes Jr.\ and Real in \cite{aurichi2024topological}, the class of edge-end spaces forms a proper subclass of the end-spaces, so a different and more selective characterisation is needed.

\section{Tree-cut decompositions and $\omega$-edge blocks}
\label{sec_2}
Our terminology about graphs -- especially about connectivity, spanning trees, cuts and bonds -- follows the textbook \cite{diestel2015book}. 

\subsection{Highly connected vertex sets in uncountable $k$-connected graphs}
G.~Dirac observed that every $2$-connected graph $G$ of uncountable regular cardinality $\kappa$ contains a pair of vertices $v\neq w$ with $\kappa$ independent paths between them (see \cite[\S 9]{halin1978simplicial}).
Dirac's assertion is equivalent to $G$ containing a subdivision of $K_{2,\kappa}$, and was generalised in this form to higher connectivity by R.~Halin \cite{halin1978simplicial} as follows:

\begin{thm}[Halin]
\label{thm_HalinTKnkappa}
Let $\kappa$ an uncountable regular cardinal, and fix $k \in \NN$.
Then every $k$-connected graph of size at least $\kappa$ contains a subdivision of $K_{k,\kappa}$.
\end{thm}

Halin's original proof is not easy and uses his theory of simplicial decompositions. The following is a new, elementary proof of Halin's theorem, that only relies on Zorn's lemma and the defining property of a regular cardinal.

We shall need the following concept: Let $W$ be some set of vertices. 
An \emph{external $k$-star attached to} $W$ is a subdivided $k$-star with precisely its leaves in~$W$ (and all other vertices outside of $W$). Its set of leaves is its \emph{attachment set}.
The \emph{interior} of an external star attached to~$W$ is obtained from it by deleting~$W$, i.e.\ its leaves.
We call a collection of external stars attached to~$W$ \emph{internally disjoint} if all its elements have pairwise disjoint interior.

\begin{proof}
Let $\kappa$ be regular uncountable, and $G=(V,E)$ be a $k$-connected graph of size at least $\kappa$. Fix an arbitrary, countably infinite set of vertices $U_0$ in $G$. We recursively construct an increasing sequence $(\,U_i \colon i \in \NN\,)$ of sets of vertices in $G$ as follows.
If $U_{i-1}$ is already defined, use Zorn's lemma to choose an inclusion-wise maximal (potentially empty) collection $\cC_i$ consisting of internally disjoint, external $k$-stars in~$G$ attached to~$U_{i-1}$, and let $U_{i} := U_{i-1} \cup V[\,\bigcup \cC_i\,]$. 

We claim that $U^* = \bigcup_{i \in \NN} U_i = V$. Otherwise, pick $v \in V \setminus U^*$. Since $G$ is $k$-connected, by Menger's theorem there is an external $k$-star attached to~$U^*$ with center $v$ and leaves $\Set{v_1,v_2,\ldots,v_k } \subset U^*$. For each $n \leq k$, let $i_n \in \NN$ be the least integer such that $v_n \in U_{i_n}$. Then for $i:= \max \, \{i_1,\ldots,i_k\}$, our external $k$ star already attaches to $U_i \supseteq \Set{v_1,v_2,\ldots,v_k }$, contradicting the maximality of~$\cC_{i+1}$.

Thus, $V = \bigcup_{i \in \NN} U_i$, and since $|V| \geq \kappa$ is regular uncountable, it follows that there is a smallest $i \in \NN$ such that $|U_i| \geq \kappa$. Note that $i > 0$. Then $|U_{i-1}|< \kappa$ and so $\cC_{i}$ consists of at least $\kappa$ internally disjoint, external $k$-stars attached to~$U_{i-1}$. Moreover, since $U_{i-1}$ consists of fewer than $\kappa$ vertices, it also has fewer than $\kappa$ distinct finite subsets. By the pigeon hole principle for regular cardinals there is a subset $\cC \subset \cC_i$ of size $\kappa$ that all have the same attachment set. Since the members of $\cC$ are internally disjoint, it follows that $\bigcup \cC$ forms the desired subdivided $K_{k,\kappa}$.
\end{proof}

\subsection{Finitely separating spanning trees}

Two vertices of an infinite graph $G$ are said to be \emph{finitely separable} in $G$ if there is a finite set of edges of $G$ separating them in $G$. Let $x \sim y$ whenever $x$ and $y$ are \emph{not} finitely separable, an equivalence relation on $V(G)$. The resulting equivalence classes are the \emph{$\omega$-edge blocks} of $G$.  
If every pair of vertices in $G$ is finitely separable, i.e.\ if all $\omega$-edge blocks are trivial, then $G$ itself is said to be \emph{finitely separable}. A spanning tree $T$ of $G$ is called \emph{finitely separating} if all its fundamental cuts are finite.

The following natural result was established only quite recently:

\begin{thm}[Kurkofka]
\label{thm_finsepspanningtree}
A connected graph is finitely separable if and only if it has a finitely separating spanning tree.
\end{thm}

Kurkofka reduced Theorem~\ref{thm_finsepspanningtree} in \cite[Theorem 5.1]{kurkofka2020every} to an earlier result \cite[Theorem~6.3]{bruhn2006duality} of Bruhn and Diestel about the topological cycle space of infinite graphs, which itself relies on further non-trivial results. 
In the following, I give an elementary proof for Theorem~\ref{thm_finsepspanningtree} (yielding, via \cite[Lemma~8.1]{kurkofka2020every}, also an elementary proof of the mentioned theorem by Bruhn and Diestel). We need a routine lemma.

\begin{lemma}
\label{lem_finseparatingfinitesets}
Let $G$ be a finitely separable graph, and let $A,B$ be disjoint, finite, connected sets of vertices in $G$. Then $G$ has a finite bond separating $A$ from $B$.
\end{lemma}
\begin{proof}
If there is no finite cut separating $A$ from $B$, then by Menger's theorem there are infinitely many edge-disjoint $A-B$ paths. Since $A$ and $B$ are finite, 
infinitely many of these paths start and end in the same vertex of $A$ and $B$ respectively, contradicting that $G$ was finitely separable. 
Now take a minimal such cut $F$ separating $A$ from $B$. Since $A$ and $B$ are connected, each of $A$ and $B$ is included in a connected component of $G-F$, and then it readily follows that this minimal cut $F$ is, in fact, a bond.
\end{proof}

\begin{proof}[Proof of Theorem~\ref{thm_finsepspanningtree}]
Let $G$ be a connected, finitely separable graph. Without loss of generality, we may assume that $G$ is 2-connected (otherwise, choose finitely separating spanning tree in each block, and consider their union).
By Theorem~\ref{thm_HalinTKnkappa} for $\kappa=\aleph_1$, every uncountable, $2$-connected graph has two vertices that are joined by uncountably many, internally disjoint paths, so fails to be finitely separable. Hence, $G$ is countable. Fix an enumeration $\set{v_n}:{n \in \NN}$ of $V(G)$.

We build an increasing sequence of finite subtrees $T_n$ in $G$ and an increasing sequence of finite sets of edges $E_n$ of $G$ such that $G_n = G - E_n$ is connected, $T_n \subseteq G_n$, and each edge of $T_n$ is a bridge of $G_n$.

Let $T_0 = \Set{v_0}$, and let $E_0 = \emptyset$. For the induction step, suppose $T_n$ and $E_n$  have already been constructed. Let $v^*$ be the first vertex in our enumeration of $G$ not yet included in $T_n$. Let $e_{n+1}$ be the first edge on a shortest $T_n - v^*$ path $P_n$ in $G_n$. Since $G_n \subset G$ is finitely separable, Lemma~\ref{lem_finseparatingfinitesets} yields a finite bond $F_{n+1}$ in $G_n$ separating $V(T_n)$ from $V(P_n) - V(T_n)$ (*). Then $e_{n+1} \in F_{n+1}$. Define $T_{n+1} = T_n + e_{n+1}$ and $E_{n+1} = E_n \cup (F_{n+1} - e_{n+1})$. Then $F_{n+1}$ witnesses that $e_{n+1}$ is a bridge of the connected graph $G_{n+1} = G - E_{n+1}$. This completes the induction step.

Then $T = \bigcup_{n \in \mathbb{N}} T_n$ is a spanning tree of $G$: It is clearly a tree. It is also spanning, as in each step, the distance from $T_n$ to the next fixed target $v^*$ strictly decreases by property (*). It remains to show that $T$ has finite fundamental cuts: Let $e_n$ be an arbitrary edge of $T$, and $T_1, T_2$ the two components of $T - e_n$. Since $e_n$ was a bridge of $G_n$, there are two components $C_1,C_2$ of $G_n - e_n$. Since $T \subseteq G_n$, we have $T- e_n \subseteq G_n-e_n$, and so (possibly after reindexing) $T_i$ spans $C_i$. But then
$$E(T_1,T_2) = E(C_1,C_2) \subseteq E_n + e_n,$$ and the latter set is finite.
\end{proof}

\subsection{Tree-cut decompositions}
Let $G$ be a graph, $T$ a tree, and let $\cX = \set{X_t}:{t \in T}$ be a partition of $V(G)$ into non-empty sets indexed by the nodes of $T$.\footnote{Some authors also allow empty parts which is sometimes useful for obtaining canonical objects.} The pair $\mathcal{T} = (T,\cX)$ is called a \emph{tree-cut decomposition} of $G$, and the sets $X_t$ are its \emph{parts}. We say that $(T,\cX)$ is a tree-cut decomposition \emph{into} these parts. 

If $(T, \cX)$ is a tree-cut decomposition, then we associate with every edge $e=t_1t_2 \in E(T)$ its \emph{adhesion set} $X_e := E_G(\bigcup_{t \in T_1} X_t , \bigcup_{t \in T_2} X_t)$ where $T_1$ and $T_2$ are the two components of $T - t_1t_2$ with $t_1 \in T_1$ and $t_2 \in T_2$. Clearly, $X_e$ is a cut in $G$. A tree-cut decomposition has \emph{finite adhesion} if all its adhesion sets are finite. 

\begin{thm}[Kurkofka]
\label{thm_treebond_finiteadhesion}
Every connected graph $G$ has a tree-cut decomposition $(T,\cX)$ of finite adhesion into its $\omega$-edge-blocks such that for all $t_1t_2 \in E(T)$ there exists an $X_{t_1}-X_{t_2}$ edge in $G$. 
\end{thm}

We repeat Kurkofka's argument from \cite{kurkofka2020every} for convenience of the reader.

\begin{proof}
Let $G$ be a connected graph. Consider the graph $\tilde{G}$ defined on the collection of $\omega$-edge blocks, i.e.\ on the equivalence classes of $\sim$, by declaring $XY$ an edge whenever $X \neq Y$ and there is an $X-Y$ edge in $G$. Note that the graph $\tilde{G}$ is a simple, connected graph that is finitely separable.

By Theorem~\ref{thm_finsepspanningtree}, there is a finitely separating spanning tree $T$ of $\tilde{G}$. This spanning tree $T$ of $\tilde{G}$ translates to a tree-cut decomposition $(T,\cX)$ of $G$ where each $X_t$ is the $\omega$-edge block of $G$ corresponding to $t \in V(\tilde{G})$.  By construction, it satisfies the property that for all $t_1t_2 \in E(T)$ there exists an $X_{t_1}-X_{t_2}$ edge in $G$. 

It remains to show that $(T,\cX)$ has finite adhesion. Since all the fundamental cuts of $T$ in $\tilde{G}$ are finite by choice of $T$, it suffices to show that if a bipartition $(A,B)$ gives rise to a finite cut of $\tilde{G}$, then the bipartition 
 $(\bigcup A, \bigcup B)$ yields a finite cut of $G$ ($\bigcup A \subset V(G)$ is the set of vertices given by the union of all edge blocks in $A$). 
Between every two distinct $\omega$-edge blocks $U$ and $W$ of $G$ there are only finitely many edges, because any single $u \in U$ is separated from $w \in W$ by a finite bond of $G$ and then $U$ and $W$ must respect this finite bond. Hence, the finitely many $A-B$ edges in $\tilde{G}$ give rise to only finitely many $(\bigcup A, \bigcup B)$ edges in $G$, and these are all $(\bigcup A, \bigcup B)$ edges in $G$. 
 \end{proof}
 
We remark that given an infinite cardinal $\kappa$, one could also consider the $\kappa$-edge blocks of a graph -- maximal sets of vertices that cannot be separated from each other by the deletion of fewer than $\kappa$ edges. Then the results of the previous two sections generalise mutatis mutandis from $\omega$-edge blocks  to $\kappa$-edge blocks as long as $\kappa$ is a \emph{regular} cardinal. But we do not need this observation in the following.

\medskip
A \emph{region} of a graph $G$ is any connected subgraph $C$ with finite \emph{boundary} $\partial C := \set{xy \in E(G)}:{x \in C,\; y \notin C}$.
Given a tree-cut decomposition  $\mathcal{T}=(T,\cX)$ of a graph~$G$ as in Theorem~\ref{thm_treebond_finiteadhesion}, we conclude this section with a lemma how regions translate from $T$ to $G$ and back again:

\begin{lemma}
\label{lem_connected}
Let $\mathcal{T}=(T,\cX)$ be a tree-cut decomposition of a connected graph~$G$ as in Theorem~\ref{thm_treebond_finiteadhesion}. Then the following assertions hold:
\begin{enumerate}
\item For every region $C'$ of $T$, also $G[\bigcup C']$ is a region of $G$.
\item For every region $C$ of $G$ there exist finitely many, pairwise disjoint regions $C'_1,\ldots,C'_n$ of $T$ such that
$C = G[\bigcup C'_1 \cup \cdots \cup \bigcup C'_n]$.
\end{enumerate}
\end{lemma}

\begin{proof} 
(1) Let $C'$ be a region of $T$ with boundary $F'$. Consider $F =  \bigcup \set{X_e}:{e \in F'} \subset E(G)$. Since $(T,\cX)$ has finite adhesion, the set $F$ is finite. 
Using that $G$ contains for all $\sim$-equivalent vertices $x$ and $y$ an $x-y$ path avoiding the finitely many edges in $F$, it follows that each $X_t$ with $t \in C'$ belongs to a single component of $G-F$. Using that for every $t_1t_2 \in E(T)$ there exists an $X_{t_1}-X_{t_2}$ edge in $G$, it follows that  $G[\bigcup C']$ is connected subset of $G-F$, so included in a component $C$ of $G-F$. Moreover, since any $X_{t}$ with $t \notin C'$ is separated from $\bigcup C'$ by $F$, it follows that $G[\bigcup C'] = C$. Since $F$ is finite, this component is a region of $G$.

(2) Let $C$ be a region of $G$ with boundary $F$. Since each $G[X_t]$ is disjoint from the finite cut $F$, every $f \in F$ belongs to at least one adhesion set $X_{f'}$ of $(T,\cX)$. 
Define $F' = \set{f'}:{f \in F}$. Then for each component $C'$ of $T-F'$, we know by (1) and the construction of $F'$ that $G[\bigcup C']$ is a connected subgraph of $G-F$. Since $T-F'$ has only finitely many components, it follows that $V(C)$ is a finite union of subgraphs of the form $G[\bigcup C']$, and the result follows.
 \end{proof}

\section{Topological results on edge-end spaces}
\label{sec_3}
\subsection{Background on topological graphs}
\label{subsec_background}

We begin by introducing the spaces $\Omega(G)$ and $\Omega_E(G)$ as well as $|G| = V(G) \cup \Omega(G)$ and $\|G\| = V(G) \cup \Omega_E(G)$ formally.

If $X \subseteq V$ is a finite set of vertices and $\omega \in \Omega(G)$ is a vertex-end, 
there is a unique component of $G-X$ that contains a tail of every ray in $\omega$, which we denote by $C(X,\omega)$. 
Then $\omega$ \emph{lives} in the component $C(X,\eps)$. 
Let $\Omega(X,\omega)$ denote the set of all ends that live in $C(X,\omega)$ and put $\hat{C}(X,\omega) = C(X,\omega) \cup \Omega(X,\omega)$.
The collection of singletons $\{v\}$ for $v \in V$ together with all sets of the form $\hat{C}(X,\omega)$ for finite $X\subseteq V(G)$ and $\omega \in \Omega(G)$ forms a basis for a Hausdorff (but not necessarily compact) topology on $|G| = V \cup \Omega$. With the corresponding subspace topology, $\Omega(G)$ is the \emph{end space} of $G$.

If $F \subseteq E$ is a finite set of edges and $\omega \in \Omega_E$ is an edge-end, 
there is a unique component of $G-F$ that contains a tail of every ray in $\omega$, which we denote by $C(F,\omega)$. Note that $C(F,\omega)$ is a region according to our earlier terminology.
We say that $\omega$ \emph{lives} in the region $C(F,\omega)$. 
An edge end $\omega$ is \emph{edge-dominated by a vertex $v$} if for every finite set of edges  $F$, the vertex $v$ belongs to $C(F,\omega)$. 
Let $\Omega_E(F,\omega)$ denote the set of all ends that live in $C(F,\omega)$. 
The collection of all $\Omega_E(C)$ for all regions $C$ of $G$ forms a basis for a Hausdorff topology on $\Omega_E(G)$.  With this topology, $\Omega_E(G)$ is the \emph{edge-end space} of $G$. 
There is also a natural way to extend the latter topology to a topology on $\|G\| = V(G) \cup \Omega_E(G)$.
If $C$ is any component of $G-F$, we write $\Omega_E(C)$ for the set of edge-ends $\omega$ of $G$ with $C(F,\omega) = C$, and abbreviate $\| C \| = C \cup \Omega_E(C)$. 
The collection of all $\|C\|$ for all regions $C$ of $G$ forms a basis for a topology on $\|G\| = V(G) \cup \Omega_E(G)$. 

If $G$ is connected, then $\|G\|$ is compact but generally no longer Hausdorff: for example, two vertices belonging to the same $\omega$-edge block cannot be separated by open sets in $\|G\|$ (in particular, only the finite degree vertices form open singleton sets in $\|G\|$).  
In this paper, we shall meet the full $\|G\|$ only on trees $\|T\|$, in which case we always have a compact Hausdorff space. 
In fact, $\|T\|$ is homeomorphic to the path space topology $\cP(T)$, see \cite{pitz2023characterising}.

\subsection{Displaying edge ends by tree-cut decompositions}
We now consider how the edge-ends of a graph $G$ interact with a  tree-cut decomposition $\mathcal{T}=(T,\cX)$ of finite adhesion. 
As every edge $e=t_1t_2 \in E(T)$ induces a finite cut $F_e := E_G(\bigcup_{t \in T_1} X_t , \bigcup_{t \in T_2} X_t)$ in $G$, any edge-end of $G$ has to choose one component $T_1$ or $T_2$ of $T-e$, and we may visualise this decision by orienting $e$ accordingly. Then for a fixed end, all the edges point either towards a unique node or towards a unique end of $T$. In this way, each edge-end of $G$ \emph{lives} in a part of $\mathcal{T}$ or \emph{corresponds} to an end of $T$, and we may encode this correspondence by a map $\varphi_\mathcal{T}  \colon \Omega_E(G) \to \|T\|$. 

We say the tree-cut decomposition \emph{distinguishes} all edge-ends if $\varphi_\mathcal{T}$ is injective; and it distinguishes all edge-ends if $\varphi_\mathcal{T}$ \emph{homeomorphically}  if $\varphi_\mathcal{T}$ is a topological embedding into $\|T\|$. 

\begin{thm}
\label{thm_displayingedgeends}
For a connected graph~$G$, the tree-cut decomposition $\mathcal{T}=(T,\cX)$ from Theorem~\ref{thm_treebond_finiteadhesion} homeomorphically distinguishes all edge-ends of $G$. Moreover, 
\begin{enumerate}
\item $\varphi_\mathcal{T}$ restricts to a bijection between the undominated edge-ends of~$G$ and the ends of~$T$, and 
\item  $\varphi_\mathcal{T}$ restricts to an injection from the dominated edge-ends of~$G$ to the nodes of $T$ such that the vertices in $X_{\varphi_\mathcal{T}(\omega)}$ are precisely the vertices edge-dominating the end $\omega$.
\end{enumerate}
\end{thm}

\begin{proof}
We begin with the following useful assertion:
\begin{clm}
\label{claim_basis}
The collection of preimages $\varphi^{-1}(\|C'\|)$, where $C'$ is a region of $T$, forms a basis for $\Omega_E(G)$. 
\end{clm}
To see the claim, first note that each such preimage is open: Indeed, by definition of $\varphi_T$ we have
\begin{equation}
\label{eqn1}
\tag{$*$}
\Omega[\bigcup C'] = \varphi^{-1}(\|C'\|),
\end{equation}
and $\Omega[\bigcup C']$ is open in $\Omega_E(G)$ since $G[\bigcup C']$ is a region of $G$ by Lemma~\ref{lem_connected}(1).
Now let $C$ be region in $G$ inducing a basic open set $\Omega_E(C)$ in $\Omega_E(G)$. By Lemma~\ref{lem_connected}(2), there are finitely many, pairwise disjoint regions $C'_1,\ldots,C'_n$ of $T$ such that $C = G[\bigcup C'_1 \cup \cdots \cup \bigcup C'_n]$. By (\ref{eqn1}) it follows 
$$\Omega_E(C)= \Omega_E(\bigcup C'_1) \cup \cdots \cup \Omega_E(\bigcup C'_n) = 
\varphi^{-1}(\|C'_1\|) \cup \cdots \cup \varphi^{-1}(\|C'_n\|),$$
which implies that preimages of regions in $T$ form a base of $\Omega_E(G)$ as claimed.

\medskip
Then $\varphi$ is a topological embedding: it is injective, since for any $\omega \neq \omega'$ there is a region $C$ of $G$ containing $\omega$ but not $\omega'$. By the claim, there is a region $C'$ in $T$ with $\omega \in  \varphi^{-1}(\|C'\|) \subseteq \Omega_E(C) \not\ni \omega'$ and hence $\varphi(\omega) \in C' \not\ni \varphi(\omega')$. 
Furthermore, Claim~\ref{claim_basis} clearly implies that $\varphi$ is a homeomorphism onto its image.

To see the moreover assertions, we first observe that if $\varphi(\omega) =: t \in V(T)$, then $\omega$ is edge-dominated by all vertices in $X_{t}$. For this, consider an arbitrary region $C$ in $G$ in which $\omega$ lives. By the claim, there is a region $C'$ of $T$ such that $\omega \in \varphi^{-1}(\|C'\|) \subset \Omega_E(C)$. But then $t = \varphi(\omega) \in C'$ implies $X_t \subseteq G[\bigcup C'] \subseteq C$, so $X_t$ belongs to every such region $C$, implying the observation.

Let $\omega \in \Omega_E(G)$. Since $(T,\cX)$ has finite adhesion, it is clear that $\phi(\omega) \in \Omega(T)$ implies that $\omega$ is undominated. Hence we know so far that $\varphi$ is injective, and maps undominated ends to $\Omega(T)$ and edge-dominated edge-ends to $V(T)$, giving (2). To complete the proof of (1) it remains to show that $\varphi$ maps onto $\Omega(T)$. So let $R=t_1t_2t_3\ldots$ be an arbitrary ray in $T$. For each $n \in \NN$ pick a vertex in $x_n \in X_{t_n}$. By the Star-Comb Lemma \cite[Lemma~8.2.2]{diestel2015book}, there is an infinite star or an infinite comb $H$ in $G$ attached to $\set{x_n}:{n \in \NN}$. Since $(T,\cX)$ has finite adhesion, we cannot get a star. So $H$ is a comb. But then the spine of $H$ belongs to an edge-end $\omega$ with $\varphi(\omega)$ being mapped to the end of $T$ containing $R$.
\end{proof}

The representation theorem from the introduction is now a simple consequence.

\representation*

\begin{proof}
By Theorem~\ref{thm_displayingedgeends}, every edge-end space is homeomorphic to a subspace $X \subset \|T\|$ for some graph-theoretic tree $T$ such that $\Omega_E(T) \subseteq X$. 

Conversely, suppose we are given such a subspace $X \subset \|T\|$ with $\Omega_E(T) \subseteq X$. We will create a graph $G_X$ with $V(G_X) = V(T)$ with $X \cong \Omega_E(G)$ by carefully adding additional edges to $T$.
Let us pick an abitrary root of $T$. We may assume that each $x \in X \cap V(T)$ has infinite degree in $T$: Otherwise, simply add some new children of $x$ as leaves to $T$ (which changes $T$ but not $X$).
For each $x \in X \cap T$, select infinitely many distinct children $t_n$ ($n \in \NN$) of $x$ and insert an edge between $t_n$ and $t_{n+1}$ for all $n \in \NN$, resulting in a ray $R_x=t_0t_1t_2\ldots$. Call the resulting graph $G_X$ and let $\omega_x$ be the end of $G_X$ containing the ray $R_x$. Then 
 $T$ is a finitely separating spanning tree of $G_X$. 
 From Theorem~\ref{thm_displayingedgeends} we know that $T$ homeomorphically displays the edge-ends of $T$, i.e.\ $\varphi \colon \Omega_E(G_X) \to \|T\|$ is a topological embedding with image $X$. Thus, $X$ and $\Omega_E(G_X) $ are homeomorphic, which concludes the proof.
\end{proof}

For (vertex-)ends, the question whether every graph admits a tree-decomposition of finite adhesion that distinguishes all ends of the underlying graph (formally posed by Diestel in 1992 \cite[4.3]{diestel1992end}), turns out to be false, see counterexamples by Carmesin \cite[\S3]{carmesin2019all} and Koloschin, Krill and Pitz \cite[\S10]{koloschin2023end}. To capture the vertex-ends of a graph, \emph{well-founded tree-decompositions} are required, see Kurkofka and Pitz \cite{kurkofkapitz_rep}.

\subsection{A metrization theorem for edge-end spaces}
\label{subsec_met}

Theorem~\ref{thm_displayingedgeends} allows us to translate topological questions of edge-end spaces to questions about subspaces of $\|T\|$. We now prove the metrization theorem announced in the introduction:

\metrization*

It is well-known that the completely ultrametrizable spaces are precisely the spaces that can be represented as (edge-)end space of a graph-theoretic tree, see e.g.~\cite{hughes2004trees}, giving $(4) \Leftrightarrow (3)$. As $(3) \Rightarrow (2) \Rightarrow (1)$ are trivial, it remains to prove $(1) \Rightarrow (4)$.

This proof relies on the following lemma. In it, we always consider \emph{rooted} trees $T$, i.e.\ trees with a special vertex $r$ called the \emph{root}. The \emph{tree-order} $\leq_T$ on $V(T)$ with root $r$ is defined by setting $u \leq_T v$ if $u$ lies on the unique path from $r$ to $v$ in $T$. Given a vertex $x$ of $T$, we write $\uc{x} = \set{v \in V(T)}:{v \geq x}$. Given an edge $e=xy$ of $T$ with $x<y$, we abbreviate $\uc{e}:= \uc{y}$. The neighbours of $x$ in $\uc{x}$ are the \emph{children} of $x$. Evidently, the collection of $\uc{x}$ induces a basis for $\Omega(T) = \Omega_E(T)$.

\begin{lemma}
\label{lem_firstcountable}
The following are equivalent for a subspace $X \subset \|T\|$:
\begin{enumerate}
	\item $X$ is first countable,
	\item every $x \in X \cap T$ has only countable many children $t$ with $\uc{t} \cap X \neq \emptyset$. 
\end{enumerate}
\end{lemma}

\begin{proof}
	For  $(1) \Rightarrow (2)$ consider some $x \in X \cap T$ with uncountably many children $t$ with $\uc{t} \cap X \neq \emptyset$,  and suppose for a contradiction that $X$ has a countable neighbourhood base $(U_n)_{n \in \NN}$ at $x$. By Hausdorffness, we have $\bigcap_{n \in \NN} U_n = \Set{x}$. However, for each $n$ there is a finite set of edges $F_n \subset E(T)$ such that the component $C_n$ of $T-F_n$ containing $x$ satisfies $C_n \subset U_n$. But then $F = \bigcup_{n \in \NN}F_n$ is countable, so some child $t$ of $x$ with $\uc{t} \cap X \neq \emptyset$ satisfies $xt \notin F$, giving $ \uc{t} \subset \bigcap_{n \in \NN} C_n \subset \bigcap_{n \in \NN} U_n$ a contradiction.
	
	Conversely, for $(2) \Rightarrow (1)$, let us fix an $x \in X$. If $x \in \Omega_E(T)$, then $x$ is represented by a unique rooted ray with edges $e_1,e_2,e_3,\ldots$. Then the regions $\uc{e_n}$ for $n \in \NN$ form a countable neighbourhood base for $x$ in $X$. And if $x \in V(T)$, then let $e_1,e_2,e_3,\ldots$ enumerate the countably edges at $x$ with $\uc{e_n} \cap X \neq \emptyset$, and let $e_0$ be the edge from $x$ to its parent (unless $x$ is the root of $X$). Write $C_n$ for the unique region of $T-\Set{e_0,e_1,\ldots,e_n}$ containing $x$; then from $(2)$ it readily follows that the $\|C_n\|$ form a countable neighbourhood base for $x$ in $X$.
		\end{proof}

\begin{proof}[Proof of Theorem~\ref{thm_met}]
It suffices to prove $(1) \Rightarrow (4)$. By the Representation Theorem~\ref{thm_rep} every edge-end space is homeomorphic to a subspace $X \subset \|T\|$ for some graph-theoretic tree $T$ such that $\Omega_E(T) \subseteq X$. 

Assuming that $X$ is first countable, we construct another tree $T'$ and show that $X$ is homeomorphic to $\Omega(T')$. By Lemma~\ref{lem_firstcountable} we know that for every $x \in X \cap V(T)$ we can enumerate its children $t'$ with $\uc{t'} \cap X \neq \emptyset$ as $t_1,t_2,t_3,\ldots$, a finite or infinite sequence. Then uncontract $x$ to a ray $R_x = s_1s_2s_3,\ldots$, connect $s_1$ to the lower neighbour of $x$, and make $t_n$ a child of $s_n$ for $n=1,2,3,\ldots$.  
	Call the resulting tree $T'$. Note that there is a natural embedding $h$ of $X$ into $\Omega(T')$: For $x \in X \cap T$ we let $h(x)$ be the end represented by the newly added ray $R_x$. And for $x \in X \cap \Omega(T)$, note that the edges of the rooted ray of $x$ in $T$ lie on a unique rooted ray in $T'$; let $h(x)$ be the corresponding end. Then it is readily seen that $h$ is a bijection between $X$ and $\Omega_E(T')$. We verify that $h$ is a homeomorphism. 
	
	To see that $h$ is continuous, suppose $h(x) = \omega$, and fix a basic open neighbourhood $\uc{t}$ of $\omega$ in $\Omega(T')$.  If $E(T)$ are cofinal in $\omega$ (i.e.\ if $x \in \Omega(T)$), then fix such an edge $e \in T[\uc{t}]$. It is easy to see that $h$ maps $\uc{e}_T$ into $\uc{t}_{T'}$. Otherwise, we may assume $t = s_n$ on $R_x$. Let $F$ consist of all edges $xt_i$ for $i \leq n$ together with the edge from $x$ to its unique predecessor in $T$. Let $C$ be the component of $T - F$ containing $x$. Then $h$ maps $C$ into $\uc{s_n}_{T'}$.
	To see that $h$ is open, consider an open set $U$ in $X$, and let $x \in U$. If $x$ is an end of $T$, then there is $e \in E(T)$ such that $x \in \uc{e}_T \cap X \subset U$, so $h(x) \in \uc{e}_{T'} \cap h(X) \subset f(U)$. If $x$ is a node of $T$, then there is a finite set of edges $F$ such that $x \in \|C\| \cap X \subset U$ for a component $C$ of $T-F$. Let $n \in \NN$ be larger than all indices of children of $x$ in $F$. Then  $h(x) \in  \uc{s_n} \cap h(X) \subset h(U) $.
\end{proof}

Finally, we mention that the graph $G_X$ constructed in the proof of Theorem~\ref{thm_rep} satisfies that $\Omega_E(G) = \Omega(G)$, giving an alternative route towards the result from {\cite{aurichi2024topological}} that the class of edge-end spaces is a subclass of the class of end spaces. In any case, this approach suggests the following open problem:

\begin{problem}
Characterize the graphs $G$ with $\Omega(G) = \Omega_E(G)$ (as topological spaces).
\end{problem}

\bibliographystyle{plain}
\bibliography{ref}
\end{document}